\documentclass[11pt]{amsart}
\usepackage{amssymb, amsbsy}
\usepackage{amsfonts, amscd}
\usepackage{mathtools}
\usepackage{geometry}
\numberwithin{equation}{section}
\usepackage{booktabs}

\usepackage[usenames,dvipsnames]{xcolor}
\usepackage{hyperref}
\hypersetup{
    colorlinks=true,
    citecolor=red,
    linkcolor=blue,
    filecolor=magenta,
    urlcolor=red,
}

\newcommand{\C}{\mathbb{C}}
\newcommand{\Q}{\mathbb{Q}}
\newcommand{\Z}{\mathbb{Z}}

\newcommand{\OO}{\mathcal{O}}

\newtheorem{thm}{Theorem}[section]

\newtheorem{lem}[thm]{Lemma}
\newtheorem{prop}[thm]{Proposition}

\theoremstyle{definition}
\newtheorem{defn}[thm]{Definition}

\newtheorem{rem}[thm]{Remark}

\title[Vector Field and weighted homogeneity]{Criteria of  isolated weighted homogeneous hypersurface singularities using Logarithmic vector fields}
\author{Jihao Liu}
\address{Department of Mathematics, Peking University, No. 5 Yiheyuan Road, Haidian District, Beijing 100871, China}
\address{Beijing International Center for Mathematical Research, Peking University, No. 5 Yiheyuan Road, Haidian District, Beijing 100871, China}
\email{liujihao@math.pku.edu.cn}

\author{Xiping Zhang}
\address{School of Mathematical Sciences and Key Laboratory of Intelligent Computing and Applications (Ministry of Education), Tongji University, Shanghai 200092, China}
\email{xzhmath@gmail.com}

\date{\today}

\begin{document}

\begin{abstract}
We prove a conjecture of da Silva Machado and Seade that characterizes weighted homogeneous isolated hypersurface singularities through the existence of a logarithmic vector field transverse to the link. For a reduced isolated hypersurface germ $(D,0)$ in $\C^{n+1}$ with $n\ge2$, or with $n=1$ and $D$ irreducible, we prove that weighted homogeneity is equivalent to the existence, in suitable coordinates, of a logarithmic vector field everywhere transverse in the real-Euclidean sense to all small links. We also prove the equivalent formulation that $(D,0)$ admits an ambient holomorphic vector field tangent to $D$ that has a non-degenerate isolated singularity at $0$. We further show that the transversality condition must be read after allowing a coordinate change: there exists a weighted homogeneous germ admitting no logarithmic field transverse to the standard round links in certain linear coordinates. The main result  of this paper was obtained by the Rethlas system.
\end{abstract}

\maketitle

\tableofcontents

\section{Introduction}\label{sec:intro}

Let $(D,0)\subset(\C^{n+1},0)$ be an isolated hypersurface singularity. da Silva Machado and Seade conjectured in \cite{MS26} that weighted homogeneity of $(D,0)$ is characterized by the existence of a holomorphic vector field tangent to $D$ and transverse to all sufficiently small links of $D$; some partial cases were obtained there. The transversality condition in this paper is the real-Euclidean one: the real part of the vector field has nonzero radial component along the link.

Our main result establishes this characterization in its coordinate-matched form, in which the Euclidean radial condition is evaluated in the same coordinates that exhibit weighted homogeneity.

\begin{thm}[Main theorem]\label{thm:main}
Let $N=n+1$, let $f\in\C\{x_0,\dots,x_n\}$ be reduced, assume that $f\in\mathfrak m^2$, and let $(D,0)=\{f=0\}\subset(\C^N,0)$ be an isolated hypersurface singularity. Assume either $n\ge2$, or $n=1$ and $D$ is irreducible. Then the following are equivalent:
\begin{enumerate}
\item $(D,0)$ is weighted homogeneous after a biholomorphic change of coordinates.
\item After a biholomorphic change of coordinates, there exists a holomorphic logarithmic vector field germ
\[
v=\sum_{i=0}^{n}\nu_i\frac{\partial}{\partial x_i}\in\operatorname{Der}(-\log D)
\]
such that
\[
2\operatorname{Re}\left(\sum_{i=0}^{n}\nu_i(P)\overline{P_i}\right)\neq0
\]
for every smooth point $P$ of $D$ on every sufficiently small Euclidean link, where the Euclidean condition is computed in those same coordinates. In other words, the vector field $\nu$ is transverse to all the link spheres.

\item There exists a holomorphic ambient vector field germ tangent to $D$
that has a non-degenerate isolated singularity at $0$, i.e.,  its linear part evaluated at $0$ is invertible.
\end{enumerate}
\end{thm}

The proof has three inputs. The implication from weighted homogeneity is given by the weighted Euler vector field. The implication from real-Euclidean transversality to an invertible ambient linear part follows from a contracting-flow argument and a first-jet estimate on the singular analytic set; the first-jet estimate uses the norm-preserving extension theorem of Agler--Kosi\'nski--McCarthy \cite{AKM22}. The implication from an invertible logarithmic linear part to weighted homogeneity uses the holomorphic version of the splitting lemma of Arnol'd \cite{Arn72} (see also \cite[Lemmma 1]{GS06b}) and the nilpotency obstruction of   Schulze \cite{Sch10} for order $\geq 3$ non-quasi-homogeneous singularities of order at least three.. 

We note that, after establishing the existence of a contracting flow, its
integration gives rise to a contracting automorphism. Consequently, weighted homogeneity follows immediately from \cite[Theorem A]{Mor24}. Nevertheless, we
retain the original AI-generated proof, since it appears that the Rethlas system
did not locate this paper and instead generated a self-contained proof of \cite[Theorem A]{Mor24} in our
special setting of isolated hypersurface singularities. The proof Rethlas system generated is different from that of \cite{Mor24}.

The coordinate clause in Theorem~\ref{thm:main}(2) is necessary. In Section~\ref{sec:fixed-coordinates} we prove Proposition~\ref{prop:fixed-coordinate-failure}, which gives a weighted homogeneous surface singularity that admits no logarithmic vector field transverse to the standard round links in a particular linear coordinate system. Thus the correct formulation is coordinate-matched: the real-Euclidean transversality condition is genuinely coordinate-dependent, and the Euclidean metric must be evaluated in the same coordinates in which the transverse logarithmic field is asserted to exist.

\begin{rem}\label{rem:rethlas}
The main result of this paper was obtained by the Danus system, a specialized agent built on Rethlas and substantially more capable for fundamental mathematical research based on the Rethlas system. See \cite{Ju+26} for an introduction to the Rethlas system. The main equivalence and its auxiliary lemmas, including the singular first-jet estimate of Lemma~\ref{lem:jet-estimate}, the contracting-flow argument of Lemma~\ref{lem:contracting-flow}, and the fixed-coordinate counterexample of Proposition~\ref{prop:fixed-coordinate-failure}, were checked by the verifiers inside the system.  The main result of the paper has also been human-verified. 
Due to the limitation of automated systems, it is possible that we have missed related references in the literature, and we welcome comments from experts.

We note that, parallel to the present paper, we provide in \cite{LZ26} a different proof of Theorem~\ref{thm:main}. Whereas the present paper was generated by AI, the paper \cite{LZ26} was written entirely by human. It is therefore interesting to compare the two papers, both from a mathematical perspective and from the perspective of mathematical writing.

The two papers essentially address the same problem and we make both papers publicly available primarily for the mathematical information it contains, and also as an illustration of the potential of AI-assisted mathematical research.
\end{rem}

\subsection*{Acknowledgements}
The first author was partially supported by the National Key R\&D Program of China \#\allowbreak 2024YFA1014400. 

The first author would like to thank the Rethlas team, namely Haocheng Ju, Jiedong Jiang, Shurui Liu, Guoxiong Gao, Yuefeng Wang, Zeming Sun, Bin Wu, Liang Xiao, and Bin Dong, for their contributions to the development of Rethlas and its customized version (the Danus system) used for the problem studied in this paper. The first author would like to thank Ruochuan Liu and Gang Tian for constant support and encouragement.

\section{Set-up and definitions}\label{sec:setup}

In this section, we fix the notation and the transversality convention used throughout the paper.

Let $N=n+1$. We consider a reduced holomorphic germ
\[
f\in\OO_{\C^N,0}=\C\{x_0,\dots,x_n\}
\]
with $f\in\mathfrak m^2$, where $\mathfrak m$ is the maximal ideal at $0$, and we put
\[
(D,0)=\{f=0\}\subset(\C^N,0).
\]
We assume that $D$ has an isolated singularity at $0$ and write $n=\dim_{\C}D$. Our standing connectedness hypothesis is that either $n\ge2$, or $n=1$ and $D$ is irreducible.

\begin{defn}\label{defn:logarithmic}
A holomorphic vector field germ
\[
v=\sum_{i=0}^{n}\nu_i\frac{\partial}{\partial x_i}
\]
is \emph{logarithmic along $D$} if it is tangent to the smooth locus $D_{\mathrm{sm}}$. Equivalently, since $f$ is reduced,
\[
v(f)=h f
\]
for some $h\in\OO_{\C^N,0}$. We write the module of such vector fields as $\operatorname{Der}(-\log D)$.
\end{defn}

Let $B_\epsilon$ be the Euclidean ball of radius $\epsilon$ centered at $0$, let $S_\epsilon=\partial B_\epsilon$, and let $L_\epsilon=D\cap S_\epsilon$ be the link. For $P\in D_{\mathrm{sm}}\cap S_\epsilon$ and $v=\sum_i\nu_i\partial_{x_i}$, set
\[
\alpha_v(P):=\sum_{i=0}^{n}\nu_i(P)\overline{P_i},\qquad
R_v(P):=2\operatorname{Re}\alpha_v(P).
\]

\begin{defn}\label{defn:real-transversality}
We say that $v$ is \emph{real-Euclidean transverse to the small links} if
\[
R_v(P)\neq0
\]
for every $P\in D_{\mathrm{sm}}\cap S_\epsilon$ and every sufficiently small $\epsilon>0$.
\end{defn}

The quantity $R_v$ depends on the chosen coordinates, equivalently on the chosen Hermitian metric. Thus the assertion in Theorem~\ref{thm:main}(2) is that there exist coordinates and a logarithmic field for which Definition~\ref{defn:real-transversality} holds in those coordinates.

We say that $(D,0)$ is \emph{weighted homogeneous} if, after a biholomorphic change of coordinates, the defining equation can be chosen weighted homogeneous for positive weights. Equivalently, in suitable coordinates there exist positive real numbers $w_i$ and a positive real number $d$ such that
\[
\sum_{i=0}^{n}w_i x_i\frac{\partial f}{\partial x_i}=d f.
\]

Condition (2) of Theorem~\ref{thm:main} is the coordinate-matched real-Euclidean transversality statement, and no isolated-zero hypothesis is included in it. Condition (3) means that there exists $v\in\operatorname{Der}(-\log D)$ whose ambient linear part $Dv(0)$ is invertible; in that case $0$ is a nondegenerate isolated zero of the ambient vector field.

\section{Connectedness of the punctured germ}\label{sec:connectedness}

In this section, we record the connectedness input used to make the radial sign constant.

\begin{prop}\label{prop:connectedness}
Let $(D,0)=\{f=0\}\subset(\C^{n+1},0)$ be a reduced isolated hypersurface singularity. If $n\ge2$, then every sufficiently small punctured representative $D\cap B_\epsilon\setminus\{0\}$ is connected. If $n=1$ and $D$ is irreducible, then every sufficiently small punctured representative is connected.
\end{prop}

\begin{proof}
Assume first that $n\ge2$. By Milnor's connectivity theorem for isolated hypersurface singularities, the link $L_\epsilon=D\cap S_\epsilon$ is $(n-2)$-connected for all sufficiently small $\epsilon$ \cite{Mil68}. In particular, $L_\epsilon$ is connected. Milnor's local conical structure theorem \cite{Mil68} identifies
\[
D\cap\bigl(B_\epsilon\setminus\{0\}\bigr)
\]
with the cone direction $(0,\epsilon]$ times $L_\epsilon$, after shrinking $\epsilon$ if necessary. Hence the punctured representative is connected.

Assume now that $n=1$ and that $D$ is irreducible. The normalization of the reduced irreducible curve germ $D$ is a single smooth disc germ, and after shrinking the normalization map is finite and is an isomorphism over $D\setminus\{0\}$, since the only singular point of $D$ is $0$. Thus $D\setminus\{0\}$ is the image of a punctured disc under a continuous finite map that is bijective over the smooth punctured curve. The punctured disc is connected, so the punctured representative of $D$ is connected.
\end{proof}

\section{A singular first-jet estimate}\label{sec:jet}

In this section, we prove the analytic estimate that converts uniform contraction on the singular germ into contraction on the Zariski tangent space.

\begin{lem}[Singular first-jet estimate]\label{lem:jet-estimate}
Let $N\ge2$, let $f\in\OO_{\C^N,0}$ be reduced with $f\in\mathfrak m^2$, and put $D=\{f=0\}$. Suppose that for some $\epsilon>0$ the analytic set
\[
V_\epsilon:=D\cap B_\epsilon
\]
is connected. Let
\[
K_\epsilon:=D\cap\overline{B_\epsilon},
\]
and let $\mathfrak m_D$ be the maximal ideal of $\OO_{D,0}$. Fix any norm on $\mathfrak m_D/\mathfrak m_D^2$. Then there exists a constant $C_\epsilon>0$ such that every holomorphic function $u$ on a neighbourhood of $K_\epsilon$ in $D$ satisfies
\[
\left\|[u-u(0)]\right\|_{\mathfrak m_D/\mathfrak m_D^2}
\le C_\epsilon\sup_{p\in K_\epsilon}|u(p)|.
\]
\end{lem}

\begin{proof}
Let $\Omega=B_\epsilon$ and $V=V_\epsilon=D\cap\Omega$. The pair $(\Omega,V)$ is a Cartan pair: $\Omega$ is a bounded pseudoconvex ball, hence a domain of holomorphy, and $V$ is the analytic subset cut out by $f$. Since $V$ is connected, the norm-preserving extension theorem of Agler--Kosi\'nski--McCarthy \cite{AKM22} gives a domain of holomorphy $G$ with
\[
V\subset G\subset\Omega
\]
such that every bounded holomorphic function on $V$ extends to a holomorphic function on $G$ with the same supremum norm.

Let $u$ be holomorphic on a neighbourhood of $K_\epsilon$ in $D$. Then $u|_V$ is bounded and
\[
\sup_V|u|\le\sup_{K_\epsilon}|u|.
\]
By the extension theorem, there exists $U\in\OO(G)$ such that $U|_V=u|_V$ and
\[
\sup_G|U|=\sup_V|u|\le\sup_{K_\epsilon}|u|.
\]
Since $0\in V\subset G$, choose $r>0$ such that the closed polydisc
\[
P_r=\{z\in\C^N\mid |z_i|\le r\text{ for all }i\}
\]
is contained in $G$. The one-variable Cauchy estimate in each coordinate direction gives
\[
\left|\frac{\partial U}{\partial x_i}(0)\right|\le r^{-1}\sup_G|U|
\]
for every $i=0,\dots,n$. Thus the ambient differential $dU_0$ is bounded by a constant, depending only on $r$ and $N$, times $\sup_{K_\epsilon}|u|$.

It remains to identify $dU_0$ with the class of $u-u(0)$ in $\mathfrak m_D/\mathfrak m_D^2$. Since
\[
\OO_{D,0}=\OO_{\C^N,0}/(f),
\]
we have
\[
\mathfrak m_D=\mathfrak m/(f),\qquad
\mathfrak m_D^2=(\mathfrak m^2+(f))/(f).
\]
Because $f\in\mathfrak m^2$, the natural map
\[
\mathfrak m/\mathfrak m^2\longrightarrow\mathfrak m_D/\mathfrak m_D^2
\]
is an isomorphism. The class $[u-u(0)]$ is represented by the image of the linear part of any ambient holomorphic extension of $u$ near $0$. The function $U$ is such an extension on $G$. If two ambient extensions are chosen, their difference vanishes on the reduced hypersurface $D$, hence lies in the principal ideal $(f)$ in the local ring; since $f\in\mathfrak m^2$, this difference has zero class in $\mathfrak m/\mathfrak m^2$. Thus the class is well-defined and is the image of $dU_0$.

All norms on the finite-dimensional vector space $\mathfrak m_D/\mathfrak m_D^2$ are equivalent. Combining this norm equivalence with the Cauchy estimate for $dU_0$ gives the desired constant $C_\epsilon$.
\end{proof}

\section{The contracting flow}\label{sec:flow}

In this section, we prove that real-Euclidean transversality forces the logarithmic vector field to have an invertible ambient linear part.

\begin{lem}\label{lem:vanishing-at-origin}
Let $f\in\OO_{\C^N,0}$ be reduced, let $D=\{f=0\}$ have isolated singular locus at $0$, and let $v$ be a holomorphic vector field germ satisfying $v(f)=h f$ for some holomorphic germ $h$. Then $v(0)=0$.
\end{lem}

\begin{proof}
Suppose that $v(0)\neq0$. Rectifying the nonvanishing holomorphic vector field $v$, we obtain local holomorphic coordinates $(y_1,\dots,y_N)$ centered at $0$ in which
\[
v=\frac{\partial}{\partial y_1}.
\]
The logarithmic equation becomes
\[
\frac{\partial f}{\partial y_1}=h f.
\]
For each fixed $y'=(y_2,\dots,y_N)$, this ordinary differential equation in the $y_1$-variable gives
\[
f(y_1,y')=u(y_1,y')\,f(0,y')
\]
for a holomorphic unit $u$, obtained by integrating $h$ in the $y_1$-direction. Thus, up to multiplication by a unit, the hypersurface germ $D$ is the product of the $y_1$-line with the hypersurface germ $\{f(0,y')=0\}\subset(\C^{N-1},0)$.

Since $D$ is singular at $0$, the factor $\{f(0,y')=0\}$ is singular at $y'=0$. Hence every point with $y'=0$ and $y_1$ sufficiently small lies in the singular locus of $D$: all derivatives in the $y_2,\dots,y_N$ directions vanish because the factor is singular at $y'=0$, and the $y_1$-derivative is $hf$, which vanishes along $D$. This gives a positive-dimensional singular locus through $0$, contradicting the isolatedness assumption. Therefore $v(0)=0$.
\end{proof}

\begin{lem}[Real transversality gives a nondegenerate zero]\label{lem:contracting-flow}
Let $N\ge2$, let $f\in\OO_{\C^N,0}$ be reduced with $f\in\mathfrak m^2$, and let $D=\{f=0\}$ have an isolated singularity at $0$. Assume that $D\setminus\{0\}$ is connected after shrinking the representative. Let
\[
v=\sum_i\nu_i\frac{\partial}{\partial x_i}\in\operatorname{Der}(-\log D)
\]
and suppose that $R_v$ is nonzero on $D\setminus\{0\}$ sufficiently near $0$. Then, after possibly replacing $v$ by $-v$, every eigenvalue of $A:=Dv(0)$ has positive real part. In particular, $A$ is invertible and $0$ is a nondegenerate isolated ambient zero of $v$.
\end{lem}

\begin{proof}
By Lemma~\ref{lem:vanishing-at-origin}, $v(0)=0$. Since $D\setminus\{0\}$ is connected and $R_v$ is continuous and never zero there, $R_v$ has constant sign after shrinking. Replacing $v$ by $-v$ if necessary, we may assume that
\[
R_v>0
\]
on the punctured representative.

Choose $\epsilon>0$ so small that
\[
K:=D\cap\overline{B_\epsilon}
\]
is a compact representative, $D\setminus\{0\}$ is smooth in $B_\epsilon$, the set $D\cap B_\epsilon$ is connected, and $R_v>0$ on $K\setminus\{0\}$. Let $\Phi_t$ be the local holomorphic flow of $v$. For a sufficiently small $T>0$, the flow is defined on a neighbourhood of $K$ for $|t|\le T$. The equation $v(f)=hf$ gives
\[
\frac{d}{dt}f(\Phi_t(x))=h(\Phi_t(x))f(\Phi_t(x)),
\]
so the flow preserves $D$ wherever it is defined. Put
\[
F=\Phi_{-T}|_D.
\] 

Let $\rho(x)=\|x\|^2$. If $\gamma_x(s)=\Phi_{-s}(x)$ for $0\le s\le T$, then for $x\in K\setminus\{0\}$,
\[
\frac{d}{ds}\rho(\gamma_x(s))=-R_v(\gamma_x(s))<0.
\]
Thus $\gamma_x(s)$ remains in $K$, and
\[
\|F(x)\|<\|x\|
\]
for every $x\in K\setminus\{0\}$. The decrease is uniform on each compact annulus. More precisely, for every $\delta\in(0,\epsilon)$ the continuous function
\[
x\longmapsto \|x\|^2-\|F(x)\|^2
\]
has a positive minimum on
\[
K\cap\{\delta\le\|x\|\le\epsilon\}.
\]
It follows that for every $\delta>0$, all $F$-orbits starting in $K$ enter $D\cap B_\delta$ after a number of iterates bounded independently of the starting point. Hence
\[
\sup_{x\in K}\|F^k(x)\|\longrightarrow0
\]
as $k\to\infty$.

We now apply Lemma~\ref{lem:jet-estimate}. Let $\ell$ be any ambient complex linear function. The functions $\ell\circ F^k$ are holomorphic on a neighbourhood of $K$ in $D$, vanish at $0$, and satisfy
\[
\sup_K|\ell\circ F^k|\le\|\ell\|\sup_K\|F^k\|\longrightarrow0.
\]
Therefore Lemma~\ref{lem:jet-estimate} gives
\[
[\ell\circ F^k]\longrightarrow0
\]
in $\mathfrak m_D/\mathfrak m_D^2$. Since $f\in\mathfrak m^2$, the natural map
\[
\mathfrak m/\mathfrak m^2\longrightarrow\mathfrak m_D/\mathfrak m_D^2
\]
is an isomorphism, and the ambient linear functions span this cotangent space. Moreover,
\[
[\ell\circ F^k]=\bigl((DF(0))^*\bigr)^k[\ell].
\]
Thus the powers of the finite-dimensional operator $(DF(0))^*$ tend to $0$, and every eigenvalue of $DF(0)$ has modulus $<1$.

Since $F$ is the time $-T$ map of $v$ and $v(0)=0$, we have
\[
DF(0)=\exp(-T A).
\]
If $\lambda$ is an eigenvalue of $A$, then $\exp(-T\lambda)$ is an eigenvalue of $DF(0)$, whence
\begin{equation}\label{eq:spectral-bound}
|\exp(-T\lambda)|<1 .
\end{equation}
Since $T>0$, the bound~\eqref{eq:spectral-bound} is equivalent to $\operatorname{Re}\lambda>0$. Hence every eigenvalue of $A$ has positive real part. In particular $A$ is invertible, and the holomorphic inverse function theorem applied to the coefficient map of $v$ shows that $0$ is an isolated nondegenerate ambient zero.
\end{proof}

\section{From a nondegenerate logarithmic zero to weighted homogeneity}\label{sec:nondegenerate}

In this section, we prove the implication from condition (3) to condition (1).

\begin{lem}\label{lem:c-to-a}
Let $f\in\C\{x_0,\dots,x_n\}$ be reduced, assume $f\in\mathfrak m^2$, and let $D=\{f=0\}$ have an isolated hypersurface singularity at $0$. Suppose that there exists a holomorphic vector field
\[
v=\sum_i\nu_i\frac{\partial}{\partial x_i}
\]
such that $v(f)=hf$ and $A:=Dv(0)$ is invertible. Then $(D,0)$ is weighted homogeneous after a biholomorphic change of coordinates.
\end{lem}

\begin{proof}
Put $N=n+1$. By the holomorphic splitting lemma \cite{Arn72}, after a biholomorphic change of coordinates we may write
\[
f=q(z)+g(y),
\]
where $z=(z_1,\dots,z_r)$, $y=(y_1,\dots,y_s)$, $r+s=N$,
\[
q=z_1^2+\cdots+z_r^2,
\]
and either $s=0$ and $g=0$, or $g\in\C\{y_1,\dots,y_s\}$ belongs to the cube of the maximal ideal and has an isolated critical point at $0$. If $s=0$, then $f=q$ is weighted homogeneous. We may therefore assume $s>0$.

Write the linear part of $v$ in block form relative to
\[
\C^N=\C_z^r\oplus\C_y^s:
\]
\[
A(z,y)=(Bz+Cy,\;D_0z+Ey),
\]
where $B\colon\C_z^r\to\C_z^r$, $C\colon\C_y^s\to\C_z^r$, $D_0\colon\C_z^r\to\C_y^s$, and $E\colon\C_y^s\to\C_y^s$ are complex linear maps. Let $h_0=h(0)$. The degree-two part of the logarithmic equation $v(f)=hf$ comes only from the linear part of $v$ acting on $q$, because $g$ has order at least three. Hence
\[
(dq)_z(Bz+Cy)=h_0q(z)
\]
for all $z$ and $y$. The left-hand side has a bilinear $z$-$y$ part, namely $(dq)_z(Cy)$, while the right-hand side has no such term. Since $q$ is nondegenerate, the identity $(dq)_z(Cy)=0$ for all $z,y$ implies $C=0$. Thus $A$ is block lower triangular. Since $A$ is invertible, both diagonal blocks $B$ and $E$ are invertible.

Now restrict the logarithmic equation to the linear subspace $z=0$. Define a vector field on the $y$-space by
\[
\theta(y):=\sum_{j=1}^{s}\nu_{y_j}(0,y)\frac{\partial}{\partial y_j}.
\]
Since $q$ and $dq$ vanish when $z=0$, we obtain
\[
\theta(g)=h(0,y)g.
\]
Thus $\theta\in\operatorname{Der}(-\log\{g=0\})$, and its linear part at $0$ is exactly $E$, which is invertible.

We use the following nilpotency obstruction, due to Schulze \cite[Proposition~4]{Sch10}: if $g\in\mathfrak m_y^3$ defines an isolated hypersurface singularity and is not weighted homogeneous, then every logarithmic vector field along $\{g=0\}$ has nilpotent linear part. Granting this, the invertibility of $E$ rules out the non-weighted-homogeneous case, so $g$ is weighted homogeneous.

For completeness we prove the obstruction. Since $g$ is not weighted homogeneous, $s\ge2$: a germ in one variable lying in $\mathfrak m_y^3$ is $y_1^k$ times a unit with $k\ge3$, hence weighted homogeneous after the coordinate change $Y=y_1u(y_1)^{1/k}$. Moreover $g$ is reduced, for a repeated irreducible factor $\varphi$ would divide every first partial derivative of $g$, and then the positive-dimensional germ $\{\varphi=0\}$ would lie in the singular locus of $\{g=0\}$, contradicting isolatedness.

By the formal structure theorem of Granger--Schulze \cite[Theorem~5.4]{GS06a}, after a formal change of coordinates the module $\operatorname{Der}(-\log\{g=0\})$ admits a minimal generating system consisting of diagonal fields $\sigma_1,\dots,\sigma_p$, each having rational eigenvalues and satisfying $\sigma_i(g)\in\Q g$, together with fields $\nu_1,\dots,\nu_q$ whose linear parts at $0$ are nilpotent. We claim $p=0$. Let $\sigma=\sum_j w_j y_j\partial_{y_j}$ be a diagonal field with $\sigma(g)=cg$ for some $c\in\C$. If $c\neq0$, then
\[
g=c^{-1}\sigma(g)=c^{-1}\sum_j w_j y_j\frac{\partial g}{\partial y_j}\in J(g),
\]
where $J(g)=(\partial g/\partial y_1,\dots,\partial g/\partial y_s)$ is the Jacobian ideal; by Saito's criterion \cite{Sai71}, $g$ is then weighted homogeneous after a biholomorphic change of coordinates, contrary to hypothesis. If $c=0$, then $\sum_j(w_jy_j)\,\partial g/\partial y_j=0$. Since $\{g=0\}$ has an isolated singularity, $J(g)$ is $\mathfrak m_y$-primary, so in the regular local ring $\C\{y_1,\dots,y_s\}$ the partial derivatives $\partial g/\partial y_1,\dots,\partial g/\partial y_s$ form a regular sequence \cite{Mat89}. Every syzygy of a regular sequence is generated by the Koszul relations, so the syzygy $\sum_j(w_jy_j)\,\partial g/\partial y_j=0$ forces each $w_jy_j$ to lie in $J(g)$. But $g\in\mathfrak m_y^3$ gives $\partial g/\partial y_j\in\mathfrak m_y^2$, hence $J(g)\subseteq\mathfrak m_y^2$; since $w_jy_j$ is a linear form, $w_jy_j=0$ for every $j$, that is $\sigma=0$. This proves $p=0$, so every generator of $\operatorname{Der}(-\log\{g=0\})$ has nilpotent linear part.

Let $L\subseteq\operatorname{End}_{\C}(\C^s)$ be the Lie algebra of linear parts at $0$ of $\operatorname{Der}(-\log\{g=0\})$. Since the singularity is isolated, every logarithmic field vanishes at $0$ by Lemma~\ref{lem:vanishing-at-origin}, so the linear part of an $\OO$-linear combination $\sum_k a_k\nu_k$ is the constant-coefficient combination $\sum_k a_k(0)(\nu_k)_0$; hence $L$ is the $\C$-span of the nilpotent linear parts $(\nu_k)_0$. By \cite[Proposition~1.4]{GS09}, the initial logarithmic Lie algebra of an isolated singularity of order at least three is solvable, so $L$ is solvable. By Lie's theorem \cite[\S4.1]{Hum72}, the elements of $L$ are simultaneously upper triangular in a suitable basis. Each $(\nu_k)_0$ is nilpotent and upper triangular, hence strictly upper triangular, so every element of the span $L$ is strictly upper triangular and therefore nilpotent. This proves the obstruction.

If $g$ has positive weights $a_j$ and weighted degree $d$, assign each quadratic variable $z_i$ the weight $d/2$. Then $q+g$ is weighted homogeneous of degree $d$. Thus $f$ is weighted homogeneous after a biholomorphic coordinate change.
\end{proof}

\section{Weighted homogeneous coordinates and the proof of the main theorem}\label{sec:assembly}

In this section, we prove the easy direction in weighted homogeneous coordinates and assemble the equivalence.

\begin{lem}\label{lem:euler}
Assume that $f$ is weighted homogeneous in coordinates $x_0,\dots,x_n$, with positive weights $w_i$ and weighted degree $d$. Then the weighted Euler vector field
\[
E=\sum_{i=0}^{n}w_i x_i\frac{\partial}{\partial x_i}
\]
is logarithmic along $D=\{f=0\}$, satisfies real-Euclidean transversality on all small links in these coordinates, and has invertible linear part at $0$.
\end{lem}

\begin{proof}
Weighted homogeneity gives
\[
E(f)=d f,
\]
so $E$ is logarithmic along $D$. For every nonzero point $P$,
\[
\alpha_E(P)=\sum_{i=0}^{n}w_i|P_i|^2
\]
is a positive real number. Hence
\[
R_E(P)=2\sum_{i=0}^{n}w_i|P_i|^2>0
\]
on every small link. The linear part of $E$ is the diagonal matrix
\[
\operatorname{diag}(w_0,\dots,w_n),
\]
which is invertible because all $w_i$ are positive.
\end{proof}

\begin{proof}[Proof of Theorem~{\rm\ref{thm:main}}]
Assume (1). Choose weighted homogeneous coordinates. By Lemma~\ref{lem:euler}, the weighted Euler vector field satisfies the real-Euclidean transversality condition in those coordinates and has invertible linear part. Thus (1) implies both (2) and (3).

Assume (2). Work in the coordinates in which (2) holds. By Proposition~\ref{prop:connectedness}, the punctured representative is connected. Lemma~\ref{lem:contracting-flow} applies to the logarithmic vector field in (2) and shows, after possibly multiplying it by $-1$, that its ambient linear part is invertible. Thus (2) implies (3).

Assume (3). Lemma~\ref{lem:c-to-a} gives that $(D,0)$ is weighted homogeneous after a biholomorphic change of coordinates. Thus (3) implies (1).

The three implications prove the equivalence of (1), (2), and (3).
\end{proof}

\section{Failure of the fixed-coordinate formulation}\label{sec:fixed-coordinates}

In this section, we prove that the implication from weighted homogeneity to real-Euclidean transversality is false if the Euclidean links are kept fixed in arbitrary linear coordinates.

\begin{lem}\label{lem:linear-part-brieskorn}
Let
\[
F(u,v,w)=u^2+v^3+w^6
\]
and let
\[
\theta=a\frac{\partial}{\partial u}
+b\frac{\partial}{\partial v}
+c\frac{\partial}{\partial w}
\]
be a holomorphic vector field germ satisfying $\theta(F)=HF$ for some holomorphic germ $H$. If the ordinary linear part of $\theta$ at $0$ is
\[
B(u,v,w)=(a_1,b_1,c_1),
\]
then there exist complex numbers $\lambda,\mu,\nu$ such that
\[
a_1=3\lambda u,\qquad
b_1=\mu u+2\lambda v,\qquad
c_1=\nu u+\lambda w.
\]
\end{lem}

\begin{proof}
Write the Taylor expansions by ordinary degree:
\[
a=a_0+a_1+a_{\ge2},\qquad
b=b_0+b_1+b_{\ge2},\qquad
c=c_0+c_1+c_{\ge2},
\]
where $a_0,b_0,c_0$ are constants, $a_1,b_1,c_1$ are linear forms, and the remaining terms have ordinary degree at least two. Write
\[
H=h_0+H_{\ge1}.
\]
The logarithmic equation is
\[
2ua+3v^2b+6w^5c=H(u^2+v^3+w^6).
\]

The degree-one part is $2ua_0$, while the right-hand side has no degree-one part. Thus $a_0=0$. The degree-two part is $2ua_1+3b_0v^2$, while the right-hand side has degree-two part $h_0u^2$. The coefficient of $v^2$ gives $b_0=0$, and then $2ua_1=h_0u^2$, so $a_1=(h_0/2)u$. The coefficient of $w^5$ in ordinary degree five is $6c_0$ on the left and $0$ on the right, because every monomial of $H(u^2+v^3+w^6)$ is divisible by $u^2$, by $v^3$, or by $w^6$. Hence $c_0=0$.

Now compare ordinary degree-three terms with no factor $u$. The term $2ua_{\ge2}$ has a factor $u$, and the contribution $H_{\ge1}u^2$ on the right has a factor $u^2$. Thus the terms with no factor $u$ in degree three come only from $3v^2b_1$ on the left and $h_0v^3$ on the right. Write
\[
b_1=B_u u+B_v v+B_w w.
\]
This comparison gives
\[
3B_vv^3+3B_wv^2w=h_0v^3.
\]
Therefore $B_v=h_0/3$ and $B_w=0$. The coefficient $B_u$ is unrestricted by this comparison.

Finally compare ordinary degree-six monomials with no factor $u$ and with $v$-degree zero or one. The term $2ua$ has a factor $u$, and $3v^2b$ has a factor $v^2$. On the right, the monomial $w^6$ can be contributed only by $h_0w^6$, and the monomial $vw^5$ cannot occur because every summand is divisible by $u^2$, by $v^3$, or by $w^6$. Write
\[
c_1=C_u u+C_v v+C_w w.
\]
The relevant part of $6w^5c_1$ is
\[
6C_vvw^5+6C_ww^6.
\]
Hence $C_v=0$ and $C_w=h_0/6$. The coefficient $C_u$ is unrestricted by this comparison.

Set $\lambda=h_0/6$, $\mu=B_u$, and $\nu=C_u$. Then
\[
a_1=3\lambda u,\qquad
b_1=\mu u+2\lambda v,\qquad
c_1=\nu u+\lambda w,
\]
as claimed.
\end{proof}

\begin{prop}[Fixed-coordinate failure]\label{prop:fixed-coordinate-failure}
Let
\[
f=x_1^2+(x_2-3x_3)^3+x_3^6\in\C\{x_1,x_2,x_3\},
\]
and let $D=\{f=0\}$. Then $D$ is a reduced isolated hypersurface singularity and is weighted homogeneous after the linear coordinate change
\[
u=x_1,\qquad v=x_2-3x_3,\qquad w=x_3.
\]
However, in the displayed $x$-coordinates, there is no logarithmic vector field $V$ satisfying
\[
2\operatorname{Re}\left(\sum_{i=1}^{3}\nu_i(P)\overline{P_i}\right)\neq0
\]
on all sufficiently small standard round links of $D$.
\end{prop}

\begin{proof}
In the coordinates
\[
u=x_1,\qquad v=x_2-3x_3,\qquad w=x_3,
\]
the defining equation becomes
\[
F(u,v,w)=u^2+v^3+w^6.
\]
It is weighted homogeneous of weighted degree $6$ for weights
\[
\operatorname{wt}(u)=3,\qquad \operatorname{wt}(v)=2,\qquad \operatorname{wt}(w)=1.
\]
Its partial derivatives are $2u$, $3v^2$, and $6w^5$, whose common zero is only the origin. Hence $D$ has an isolated hypersurface singularity at $0$. The same derivative calculation shows that $F$, and therefore $f$, is reduced: a repeated factor would divide $F$ and all of its partial derivatives, but those partial derivatives have no common nonunit factor. Also $f\in\mathfrak m^2$.

Assume, for contradiction, that there exists a holomorphic vector field germ
\[
V=\sum_{i=1}^{3}\nu_i\frac{\partial}{\partial x_i}
\]
such that $V(f)=hf$ and such that the fixed-coordinate real-Euclidean radial component $R_V$ is nonzero on all sufficiently small links. Since $\dim_{\C}D=2$, Proposition~\ref{prop:connectedness} gives that $D\setminus\{0\}$ is connected after shrinking. Thus $R_V$ has constant sign there. Replacing $V$ by $-V$ if necessary, assume $R_V>0$ on the punctured representative. Lemma~\ref{lem:contracting-flow}, applied in the $x$-coordinates, shows that every eigenvalue of the linear part $A=DV(0)$ has positive real part.

Let
\[
L(u,v,w)=(u,v+3w,w)
\]
be the linear map from the weighted coordinates to the original $x$-coordinates. Pull $V$ back by $L$ to a logarithmic vector field
\[
\theta=a\frac{\partial}{\partial u}
+b\frac{\partial}{\partial v}
+c\frac{\partial}{\partial w}
\]
for $F$. Linear coordinate change conjugates the linear parts of $V$ and $\theta$, so they have the same eigenvalues. By Lemma~\ref{lem:linear-part-brieskorn}, the linear part $B=D\theta(0)$ has the form
\[
B(u,v,w)=(3\lambda u,\mu u+2\lambda v,\nu u+\lambda w)
\]
for some complex numbers $\lambda,\mu,\nu$. Therefore the eigenvalues of $A$ are $3\lambda$, $2\lambda$, and $\lambda$. Since every eigenvalue of $A$ has positive real part, we have $\operatorname{Re}\lambda>0$.

Let $p,q\in\mathbb R$ not both zero. For small positive real $t$, choose $u_t$ satisfying
\[
u_t^2=-t^3p^3-t^6q^6
\]
and set
\[
y_t=(u_t,tp,tq).
\]
Then $F(y_t)=0$. For the two choices of $(p,q)$ used below, $p\neq0$, so $y_t$ is a smooth point of $\{F=0\}$ for all sufficiently small $t$. Moreover $u_t=O(t^{3/2})$. Put
\[
x_t=L(y_t).
\]

We compute the leading term of $R_V(x_t)$. The nonlinear terms of $\theta$ are $O(t^2)$, and after applying $L$ and pairing with $x_t=O(t)$ they contribute $O(t^3)$. The $\mu u$ and $\nu u$ terms of the linear part contribute $O(t^{3/2})$ to the vector component and hence $O(t^{5/2})$ to the radial pairing. The order-$t^2$ term therefore comes only from the diagonal $(2\lambda v,\lambda w)$ part. Since
\[
L(0,2\lambda tp,\lambda tq)=(0,\lambda t(2p+3q),\lambda tq)
\]
and
\[
L(0,tp,tq)=(0,t(p+3q),tq),
\]
we obtain
\[
R_V(x_t)=2t^2\operatorname{Re}\bigl(\lambda\,\beta(p,q)\bigr)+O(t^{5/2}),
\]
where
\[
\beta(p,q)=(2p+3q)(p+3q)+q^2.
\]
For $(p,q)=(1,0)$, we have $\beta(p,q)=2$. For $(p,q)=(-9/4,1)$, we have
\[
\beta(p,q)=-\frac18.
\]
Because $\operatorname{Re}\lambda>0$, the first family has $R_V(x_t)>0$ for all sufficiently small $t$, while the second family has $R_V(x_t)<0$ for all sufficiently small $t$. This contradicts the sign-constancy already obtained from connectedness.

Therefore no logarithmic vector field in the displayed $x$-coordinates satisfies fixed-coordinate real-Euclidean transversality on all small standard round links. Since the same germ is weighted homogeneous in the $(u,v,w)$-coordinates, transversality in Theorem~\ref{thm:main}(2) is genuinely coordinate-dependent, and the coordinate-matched formulation is necessary.
\end{proof}

\section{Final remark}\label{sec:final-remark}

Since multiplying a weighted Euler field by $i$ preserves the nonvanishing of $\alpha_v$ but makes $R_v$ identically zero, the complex-hyperplane condition $\alpha_v(P)\neq0$ is strictly weaker than real-Euclidean transversality. In fact the analogous characterization fails, as the following shows.

\begin{prop}\label{prop:complex-hyperplane-failure}
Let $D$ be the irreducible plane branch parametrized by $(x,y)=(t^4,t^6+t^7)$. Then $D$ is not weighted homogeneous, yet it admits a logarithmic vector field $v\in\operatorname{Der}(-\log D)$ with $\alpha_v(P)\neq0$ on every sufficiently small link.
\end{prop}

\begin{proof}
The map $t\mapsto(t^4,t^6+t^7)$ is the normalization of $D$, so $D$ is irreducible and reduced. Its value semigroup $S=\operatorname{ord}_t\OO_{D,0}\subseteq\Z_{\ge0}$ contains $4=\operatorname{ord}_t x$ and $6=\operatorname{ord}_t y$, and also $13$, because
\[
y^2-x^3=(t^6+t^7)^2-t^{12}=2t^{13}+t^{14}
\]
has $t$-order $13$. A weighted homogeneous irreducible plane branch is, in suitable coordinates, defined by a polynomial $Y^p-\lambda X^q$ with $\gcd(p,q)=1$, so its value semigroup is generated by the two coprime numbers $p$ and $q$. The semigroup $S$ is not two-generated: its least positive element is the multiplicity $4$, so any pair of generators would contain $4$; the second generator is at most $6$, but it is neither $5$, since $6\notin\langle4,5\rangle$, nor $6$, since $\langle4,6\rangle$ consists of even numbers and so cannot contain $13$. Hence $S$ is not two-generated, and $D$ is not weighted homogeneous; equivalently, $D$ is not quasi-homogeneous in the sense of Saito \cite{Sai71}.

We now construct the field. Embed $\OO_{D,0}\subset\C\{t\}$ through the normalization, and let $c$ be the conductor of $S$, so that $t^c\C\{t\}\subseteq\OO_{D,0}$. Fix an integer $k\ge c+1$. For $g\in\OO_{D,0}$ the derivative $g'$ lowers $t$-order by one, so $t^kg'\in t^c\C\{t\}\subseteq\OO_{D,0}$; thus $\delta:=t^k\,d/dt$ is a $\C$-derivation of $\OO_{D,0}$. Let $f\in\C\{x,y\}$ be a reduced equation of $D$, and choose ambient holomorphic lifts $a,b\in\C\{x,y\}$ of $\delta(\bar x),\delta(\bar y)$. The field $v=a\,\partial_x+b\,\partial_y$ reduces modulo $(f)$ to $\delta$; since $\delta$ annihilates the class of $f$, we obtain $v(f)\in(f)$, that is $v\in\operatorname{Der}(-\log D)$. On the normalization,
\[
\alpha_v\bigl((t^4,t^6+t^7)\bigr)=t^k\bigl(x'(t)\overline{x(t)}+y'(t)\overline{y(t)}\bigr)
=t^k\bigl(4t^3\overline{t}^{\,4}+(6t^5+7t^6)(\overline{t}^{\,6}+\overline{t}^{\,7})\bigr),
\]
whose leading term $4t^{k+3}\overline{t}^{\,4}$ has modulus $4|t|^{k+7}$. Since the normalization is an isomorphism over $D\setminus\{0\}$ after shrinking, $\alpha_v(P)\neq0$ for every $P\in D_{\mathrm{sm}}$ on every sufficiently small link.
\end{proof}

Hence the implication~$(2)\Rightarrow(1)$ of Theorem~\ref{thm:main} does not hold when $R_v(P)\neq0$ is replaced by $\alpha_v(P)\neq0$.


\end{document}